\documentclass[12pt]{amsart}

\usepackage{amsmath,amsthm,amssymb}
\usepackage{url,enumitem}
\usepackage{float,subfigure,lscape}
\usepackage{multicol,multirow,array}
\usepackage{graphicx,tikz-cd}
\usepackage[left=1.8cm,right=1.8cm, top=3cm, bottom=3cm, a4paper]{geometry}

\newtheorem{thm}{Theorem}[section]
\newtheorem{lemma}[thm]{Lemma}
\newtheorem{corollary}[thm]{Corollary}
\newtheorem{prop}[thm]{Proposition}
\newtheorem{conjecture}[thm]{Conjecture}

\theoremstyle{definition}

\newtheorem{remark}[thm]{Remark}

\newtheorem{question}[thm]{Question}

\usetikzlibrary{matrix}

\newcommand{\Ker}{\mathrm{Ker}}

\newcommand{\PSLp}{\mathrm{PSL}_2(\ZZ[\frac{1}{p}])}
\newcommand{\SLp}{\mathrm{SL}_2(\ZZ[\frac{1}{p}])}
\newcommand{\SLm}{\mathrm{SL}_2(\ZZ[\frac{1}{m}])}

\newcommand{\PSL}{\mathrm{PSL}}
\newcommand{\SL}{\mathrm{SL}}
\newcommand{\GL}{\mathrm{GL}}
\newcommand{\HH}{\mathbb{H}}
\newcommand{\ZZ}{\mathbb{Z}}

\newcommand{\RR}{\mathbb{R}}
\newcommand{\CC}{\mathbb{C}}
\newcommand{\QQ}{\mathbb{Q}}

\newcommand{\calo}{\mathcal{O}}
\newcommand{\calr}{\mathcal{R}}
\newcommand{\repr}{\calr_\RR}
\newcommand{\repc}{\calr_\CC}
\newcommand{\reph}{\calr_\HH}
\newcommand{\calf}{\mathcal{F}}
\newcommand{\orbf}{\mathbf{Or}_\calf(\Gamma)}
\newcommand{\sym}{{\rm Sym}}

\newcommand{\FIN}{\mathcal{FIN}}

\newcommand{\bfE}{\mathbf{E}}
\newcommand{\Top}{\mathbf{Top}}
\newcommand{\spectra}{\mathbf{Spectra}}
\newcommand{\ko}{\mathbf{ko}}
\newcommand{\KO}{\mathbf{KO}}
\DeclareMathOperator{\hocolim}{hocolim}
\DeclareMathOperator{\Tor}{\mathrm{Tor}}

\newcommand{\calt}{\mathcal{T}}
\usepackage{tikz}
\usetikzlibrary{arrows,quotes}
\tikzstyle{blackNode}=[fill=black, draw=black, shape=circle]

\title{On the equivariant $K$- and $KO$-homology of some special linear groups}
\author{Sam Hughes}

\date{\today}

\begin{document}
\maketitle

\begin{abstract}
We compute the equivariant $KO$-homology of the classifying space for proper actions of $\textrm{SL}_3(\mathbb{Z})$ and $\textrm{GL}_3(\mathbb{Z})$. We also compute the Bredon homology and equivariant $K$-homology of the classifying spaces for proper actions of $\textrm{PSL}_2(\mathbb{Z}[\frac{1}{p}])$ and $\textrm{SL}_2(\mathbb{Z}[\frac{1}{p}])$ for each prime $p$.  Finally, we prove the Unstable Gromov-Lawson-Rosenberg Conjecture for a large class of groups whose maximal finite subgroups are odd order and have periodic cohomology.
\end{abstract}

\section{Introduction}
There has been considerable interest in the Baum-Connes conjecture, which states that for a group $\Gamma$ a certain `assembly map', from the equivariant $K$-homology of the classifying space for proper actions $\underbar{E}\Gamma$ to the topological $K$-theory of the reduced group $C^*$-algebra, is an isomorphism \cite{BaumConnesHigson1994}.  The Baum-Connes conjecture is known to hold for several families of groups, including word-hyperbolic groups, ${\rm CAT}(0)$-cubical groups and groups with the Haagerup property.  An excellent survey can be found in \cite{AparicioJulgValette20}. 

\begin{conjecture}[The Baum-Connes Conjecture]
Let $\Gamma$ be a discrete group, then the following {\emph assembly map} is an isomorphism 
\[\mu:K^\Gamma_\ast(\underbar{\emph{E}}\Gamma)\rightarrow K^{\rm top}_\ast(C^\ast_r(\Gamma)).\]
\end{conjecture}

There is also a `real' Baum-Connes conjecture which predicts that an assembly map from the equivariant $KO$-homology of $\underbar{E}\Gamma$ to the topological $K$-theory of the real group $C^\ast$-algebra is an isomorphism.  It is known that the Baum-Connes Conjecture implies the Real Baum-Connes Conjecture \cite{BK2004}.

\begin{conjecture}[The Real Baum-Connes Conjecture]
Let $\Gamma$ be a discrete group, then the following {\emph assembly map} is an isomorphism 
\[\mu_\RR:KO^\Gamma_\ast(\underbar{\emph{E}}\Gamma)\rightarrow KO^{\rm top}_\ast(C^\ast_r(\Gamma)).\]
\end{conjecture}

In spite of the interest, to date there have been very few computations of $K^\Gamma$- and $KO^\Gamma$-homology. Indeed, for $K^\Gamma$-homology there are complete calculations for one relator groups \cite{MislinNotes}, NEC groups \cite{luck2000computations}, some Bianchi groups and hyperbolic reflection groups \cite{LORS2018, Rahm:Bianchi, Rahm:Bianchi2}, some Coxeter groups \cite{DL17,SanchezGarcia:Coxeter,SanchezGarcia:PhD}, Hilbert modular groups \cite{SaldanaVelasquez2018}, $\SL_3(\ZZ)$ \cite{SanchezGarcia:SL}, and $\PSL_4(\ZZ)$ \cite{PSL4Z}.  Explicit assembly maps have also been computed for solvable Baumslag-Solitar groups \cite{PooyaValette18}, lamplighter groups of finite groups \cite{FloresPooyaValette17} and certain wreath products \cite{Pooya19,Li19}.  For $KO^\Gamma$-homology the author is aware of two complete computations; the first, due to Davis and L\"uck, on a family of Euclidean crystallographic groups \cite{DL2013}, and the second, due to Mario Fuentes-Rum\'i, on simply connected graphs of cyclic groups of odd order and of some Coxeter groups \cite{fuentesrumi2019equivariant}.

In this paper  we compute the equivariant $K$-homology of $\SLp$ and the equivariant $KO$-homology of $\SL_3(\ZZ)$.  We give the relevant background and the connection to Bredon homology in Section~\ref{Sec.bckgrd}.  

The calculation for $KO^\Gamma$-homology is of particular interest because it is (to the author's knowledge) the first computation of $KO^\Gamma_\ast$ for a property (T) group.  For background on property (T) the reader may consult the monograph \cite{BekkaHarpeValette2008}.  This interest stems from the fact that property (T) is a strong negation of the Haagerup property which implies Baum-Connes \cite{HigsonKasparov1997}.  Moreover, the (real) Baum-Connes conjecture is still open for $\SL_n(\ZZ)$ when $n\geq3$.  We note that there are counterexamples for the Baum-Connes conjecture for groupoids constructed from $\SL_3(\ZZ)$ and more generally a discrete group with property (T) for which the assembly map is known to be injective \cite{HLS02}.

\begin{thm}[Theorem~\ref{thm.KO.sl3.body}]\label{thm.KO.sl3}
Let $\Gamma=\SL_3(\ZZ)$, then for $n=0,\dots,7$ we have
\[KO^\Gamma_{n}(\underbar{\emph{E}}\Gamma)= \ZZ^8,\quad \ZZ_2^8,\quad \ZZ_2^8,\quad 0,\quad \ZZ^8, \quad 0, \quad 0, \quad 0  \]
and the remaining groups are given by $8$-fold Bott-periodicity.
\end{thm}

Applying a K\"unneth type theorem \cite[Theorem~3.6]{SanchezGarcia:Coxeter} to the isomorphism $\GL_3(\ZZ)\cong\SL_3(\ZZ)\times\ZZ_2$ on the level of Bredon homology, we obtain the following result for $\GL_3(\ZZ)$.

\begin{corollary}[Theorem~\ref{cor.KO.gl3.body}]\label{cor.KO.gl3}
Let $\Gamma=\GL_3(\ZZ)$, then for $n=0,\dots,7$ we have
\[KO^\Gamma_{n}(\underbar{\emph{E}}\Gamma)= \ZZ^{16},\quad \ZZ_2^{16},\quad \ZZ_2^{16},\quad 0,\quad \ZZ^{16}, \quad 0, \quad 0, \quad 0  \]
and the remaining groups are given by $8$-fold Bott-periodicity.
\end{corollary}

We also consider $\Gamma=\PSLp$ or $\SLp$, for $p$ a prime, computing the equivariant $K$-homology groups $K^\Gamma_n(\underbar{E}\Gamma)$.  There has been considerable interest in determining homological properties of the groups $\SLm$ and groups related to them \cite{AN1998,BE2014,Hughes2019cohomology}.  It appears, however, that even with computer based methods the problem of determining the cohomology of $\SLm$ for $m$ a product of $3$ primes is out of reach \cite{BE2014}. In Lemma~\ref{lem.BC} we give a short proof of the Baum-Connes conjecture for $\SLp$ and so we obtain the topological $K$-theory of the reduced group $C^*$-algebra of $\SLp$ as well.

\begin{thm}[Theorem~\ref{thm.equiK.PSL1overp.body}]\label{thm.equiK.PSL1overp}
Let $p$ be a prime and $\Gamma=\PSL_2(\ZZ[\frac{1}{p}])$, then $K_n^{\Gamma}(\underbar{{\emph E}}\Gamma)$ is a free abelian group with rank as given in Table~\ref{tab:PSL1pEquiK}.  Moreover, since the Baum-Connes Conjecture holds for $\Gamma$ we have $K^\Gamma_\ast( \underbar{{\emph E}}\Gamma)\cong K^\textrm{top}_\ast(C^\ast_r(\Gamma))$.
\begin{table}[h]
    \centering
\begin{tabular}{|c|c|c|c|c|c|c|}
    \hline
    & $p=2$ & $p=3$ & $ p\equiv 1\pmod{12}$& $ p\equiv 5\pmod{12}$& $ p\equiv 7\pmod{12}$ & $ p\equiv 11\pmod{12}$\\
    \hline \hline
    $n=0$ & $7$ & $6$ & $4+\frac{1}{6}(p-7)$ & $6+\frac{1}{6}(p+1)$ & $5+\frac{1}{6}(p-1)$ & $7+\frac{1}{6}(p+7)$ \rule{0pt}{11pt} \\
    \hline
    $n=1$ & $0$ & $0$ & $3$ & $1$ & $2$ & $0$\\
    \hline
\end{tabular}
    \caption{$\ZZ$-rank of the equivariant $K$-homology of the classifying space for proper actions of $\PSL_2(\ZZ[\frac{1}{p}])$.}
    \label{tab:PSL1pEquiK}
\end{table}
\end{thm}

\begin{thm}[Theorem~\ref{thm.equiK.Sl1overp.body}]\label{thm.equiK.SL1overp}
Let $p$ be a prime and $\Gamma=\SLp$.  Then $K_n^\Gamma(\underbar{\emph{E}}\Gamma)$ is additively isomorphic to the direct sum of two copies of the corresponding equivariant $K$-homology group of $\PSLp$.
\end{thm}

Finally, we will give a proof of the Unstable Gromov-Lawson-Rosenberg Conjecture for positive scalar curvature for a large class of groups whose torsion subgroups have periodic cohomology.  The statement and background concerning this conjecture is given in Section~\ref{sec.GLR}.  However, we will introduce the following notation now before the theorem statement.  We say a group $\Gamma$ satisfies:

\begin{enumerate}
    \item[(M)] If every finite subgroup is contained in a unique maximal finite subgroup.
    \item[(NM)] If $M$ is a maximal finite subgroup of $\Gamma$, then the normaliser $N_\Gamma(M)$ of $M$ is equal to $M$.
    \item[(BC)] If $\Gamma$ satisfies the Baum-Connes conjecture.
    \item[(PFS)] If all maximal finite subgroups of $\Gamma$ are odd order and have periodic cohomology.
\end{enumerate}

A large number of arithmetic groups satisfy the following theorem including many finite index subgroups $\PSLp$ for $p\equiv11\pmod{12}$ and Hilbert modular groups.  We will detail a number of additional examples in Section~\ref{sec.GLR}.

\begin{thm}[Theorem~\ref{thm.GLR.body}]\label{thm.GLR}
Let $\Gamma$ be a group satisfying (BC), (M), (NM) and (PFS).  If $\underbar{\emph{B}}\Gamma$ is finite and has dimension at most $9$, then the Unstable Gromov-Lawson-Rosenberg Conjecture holds for $\Gamma$.
\end{thm}

In Section~\ref{Sec.bckgrd} we give the relevant background on equivariant $K$ and $KO$-homology.  In Section~\ref{sec.SL3} we give the computations of the equivariant $KO$-homology for $\SL_3(\ZZ)$ and $\GL_3(\ZZ)$.  In Section~\ref{sec.KthryFuchsian} we provide auxiliary computations of the equivariant $K$-homology of Fuchsian groups.  In Section~\ref{sec.SLpKthry} we compute the equivariant $K$-homology of $\PSLp$ and $\SLp$.  Finally, in Section~\ref{sec.GLR} we prove the results about the Unstable Gromov-Lawson-Rosenberg Conjecture and give a number of examples of groups satisfying the conjecture.

\subsection*{Acknowledgements}
This paper contains material from the author's PhD thesis.  The author would like to thank his PhD supervisor Professor Ian Leary for his guidance and support.  He would also like to thank Jim Davis for pointing out Remark~\ref{remark.GLR}, Naomi Andrew, Guy Boyde, and Kevin Li for helpful conversations and the anonymous reviewer whose feedback greatly improved the exposition of this paper. This work was supported by the Engineering and Physical Sciences Research Council grant number 2127970.

\section{Preliminaries}\label{Sec.bckgrd}
In this section we introduce the relevant background from Bredon homology and its interactions with equivariant $K$- and $KO$-homology.  We follow the treatment given in Mislin's notes \cite{MislinNotes}.

\subsection{Classifying spaces for families}
Let $\Gamma$ be a discrete group.  A $\Gamma$-CW complex $X$ is a CW-complex equipped with a cellular $\Gamma$-action. We say the $\Gamma$ action is proper if all of the cell stabilisers are finite.

Let $\calf$ be a family of subgroups of $\Gamma$ which is closed under conjugation and finite intersections.  A {\it model} for the classifying space $E_\calf \Gamma$ for the family $\calf$ is a $\Gamma$-CW complex such that all cell stabilisers are in $\calf$ and the fixed point set of every $H\in\calf$ is weakly-contractible.  This is equivalent to the following universal property:  For every $\Gamma$-CW complex $Y$ there is exactly one $\Gamma$-map $Y\rightarrow E_\calf\Gamma$ up to $\Gamma$-homotopy.

In the case where $\calf=\mathcal{FIN}$, the family of all finite subgroups of $\Gamma$, we denote $E_{\mathcal{FIN}}(\Gamma)$ by $\underbar{E}\Gamma$.  We call such a space, the classifying space for proper actions of $\Gamma$.  Note that if $\Gamma$ is torsion-free then $\underbar{E}\Gamma=E\Gamma$.

\subsection{Bredon homology}
Let $\Gamma$ be a discrete group and $\calf$ be a family of subgroups.  We define the {\it orbit category} $\orbf$ to be the category with objects given by left cosets $\Gamma/H$ for $H\in\calf$ and morphisms the $\Gamma$-maps $\phi:\Gamma/H\rightarrow\Gamma/K$.  A morphism in the orbit category is uniquely determined by its image $\varphi(H)=\gamma K$ and $\gamma H\gamma^{-1}\subseteq K$; conversely, each such $\gamma\in\Gamma$ defines a $G$-map.

A {\it (left) Bredon module} is a covariant functor $M:\orbf\rightarrow\mathbf{Ab}$, where $\mathbf{Ab}$ is the category of Abelian groups.  Consider a $\Gamma$-CW complex $X$ and a family of subgroups $\calf$ containing all cell stabilisers.  Let $M$ be a Bredon module and define the {\it Bredon chain complex with coefficients} in $M$ as follows:

Let $\{c_\alpha\}$ be a set of orbit representatives of the $n$-cells in $X$ and let $\Gamma_\alpha$ denote the stabiliser of the cell $\alpha$.  The $n$th chain group is then
\[C_n:=\bigoplus_\alpha M(\Gamma/\Gamma_{c_\alpha}). \]
If $\gamma c'$ is an $(n-1)$-cell in the boundary of $c$, then $\gamma^{-1}\Gamma_c\gamma\subseteq\Gamma_{c'}$.  This defines a $\Gamma$-map $\varphi:\Gamma/\Gamma_c\rightarrow\Gamma/\Gamma_{c'}$, which in turn gives an induced homomorphism $M(\varphi):M(\Gamma/\Gamma_c)\rightarrow M(\Gamma/\Gamma_{c'})$.  Therefore, we obtain a differential $\partial:C_n\rightarrow C_{n-1}$.  Taking homology of the chain complex $(C_\ast,\partial)$ gives the {\it Bredon homology groups} $H^\calf_n(X;M)$.  A right Bredon module and Bredon cohomology are defined analogously using contravariant functors.

\subsection{Equivariant $K$-homology}
The original definition of equivariant $K$-homology used Kasparov's $KK$-theory \cite{BaumConnesHigson1994}.  There is also homotopy theoretic approach using spaces and spectra over the orbit category due to Davis-L\"uck \cite{DL1998}.  We will highlight the details we need.

Let $\Gamma$ be a discrete group.  In the context of the Baum-Connes conjecture we are specifically interested in the case where $X=\underbar{E}\Gamma$, $\calf=\mathcal{FIN}$ and $M=\repc$ the complex representation ring.  We consider $\repc(-)$ as a Bredon module in the following way:  For $\Gamma/H\in\orbf$ set $\repc(\Gamma/H):=\repc(H)$, the ring of complex representations of the finite group $H$.  Morphisms are then given by induction of representations.

We note that $\repc(\Gamma):=H_0^{\mathcal{F}}(\Gamma)=\textrm{colim}_{\Gamma/H\in\mathbf{Or}(\Gamma)}R_\CC(H)$.  In the case that $\Gamma$ has finitely many conjugacy classes of finite subgroups, $\repc(\Gamma)$ is a finitely generated quotient of $\bigoplus \repc(H)$, where $H$ runs over conjugacy classes of finite subgroups.

We now exhibit the connection between Bredon homology and $K^\Gamma_\ast(\underbar{E}\Gamma)$, the equivariant $K$-homology of the classifying space for proper actions.  Indeed, for each subgroup $H\leq\Gamma$ equivariant $K$-homology satisfies 
\[K_n^\Gamma(\Gamma/H)=K_n^{\rm{top}}(C_r^\ast(H)).\]
In the case $H$ is a finite subgroup we have $C_r^\ast(H)=\CC H$, $K_0^\Gamma(\Gamma/H)=K_0^{\rm{top}}(\CC H)=\repc(H)$, and $K_1^\Gamma(\Gamma/H)=K_1^{\rm{top}}(\CC H)=0$.  The remaining $K^\Gamma$ groups are given by $2$-fold Bott periodicity.  This allows us to view $K^\Gamma_n(-)$ as a Bredon module over $\mathbf{Orb}_{\mathcal{FIN}}(\Gamma)$.

We may use an equivariant Atiyah-Hirzebruch spectral sequence to compute the $K^\Gamma$-homology of a proper $\Gamma$-CW-complex $X$ from its Bredon homology.

\begin{thm}\emph{\cite[Page~50]{MislinNotes}}
Let $\Gamma$ be a group and $X$ a proper $\Gamma$-CW complex, then there is an Atiyah-Hirzebruch type spectral sequence
\[E^2_{p,q}:=H^{\mathcal{FIN}}_p(X;K^\Gamma_q(-))\Rightarrow K^\Gamma_{p+q}(X). \]
\end{thm}

\subsection{Equivariant $KO$-homology}
In this section we summarise the material from \cite[Section~9]{DL2013} which we will require for our calculations.  Again fixing a discrete group $\Gamma$ and $\calf=\mathcal{FIN}$, we introduce two more Bredon modules, the real representation ring $\repr(-)$, and the quaternionic representation ring $\reph(-)$.  These are defined on $\orbf$ in exactly the same way as the complex representation ring.  We have natural transformations between the functors.  Indeed for a finite subgroup $H\leq\Gamma$ we have a diagram:

\[\begin{tikzcd}
\repr(H) \arrow[r, "\nu", bend left] & \repc(H) \arrow[r, "\sigma", bend left] \arrow[l, "\rho", bend left] & \reph(H). \arrow[l, "\eta", bend left]
\end{tikzcd}\]
Note that the diagram does not commute.  For instance let $\mathbf{1}\in\repr(H)$ denote the trivial representation, then $\rho\nu(\mathbf{1})=2\cdot(\mathbf{1})$.

For a real representation $\psi$, the {\it complexification} is $\nu(\psi)=\psi\otimes\CC$.  For a complex representation $\phi$, the {\it symplectification} is $\sigma(\phi)=\phi\otimes\HH$.  Going the other way, for an $n$-dimensional quaternionic representation $\omega$, the {\it complexification} is $\eta(\omega)=\eta$ considered as $2n$-dimensional complex representation.  Similarly, for an $n$-dimensional complex representation $\phi$, the {\it realification} is $\rho(\phi)=\phi$ considered as a $2n$-dimensional real representation.  Note that any composition of the x-ification natural transformations with the same source and target is necessarily not the identity.

The situation for the equivariant $KO$-homology, denoted $KO^\Gamma_\ast(-)$, is similar to the equivariant $K$-homology but more complicated.  For a subgroup $H\leq\Gamma$ we set $KO_n^\Gamma(\Gamma/H)=KO_n^{\rm{top}}(C_r^\ast(H))$.  By \cite[Section~1.2]{BrunerGreenlees2010}, in the case that $H$ is a finite subgroup we have that
\[KO_{n}^\Gamma(\Gamma/H)=KO_{n}^{\rm{top}}(C_r^\ast(H))=\begin{cases}
\repr(H) & n=0,\\
\repr(H)/\rho(\repc(H) & n=1,\\
\repc(H)/\eta(\reph(H)) & n=2,\\
0 & n=3,\\
\reph(H) & n=4,\\
\reph(H)/\sigma(\repc(H)) & n=5,\\
\repc(H)/\nu(\repr(H)) & n=6,\\
0 & n=7,
\end{cases} \]
with the remaining groups given by $8$-fold Bott-periodicity.  For $X$ a proper $\Gamma$-space, the Atiyah-Hirzebruch spectral sequence from before now takes the form
\[E^2_{p,q}:=H^{\mathcal{FIN}}_p(X;KO^\Gamma_q(-))\Rightarrow KO^\Gamma_{p+q}(X). \]

\subsection{Spectra and homotopy}
The section gives an alternative $\Gamma$-equivariant homotopy theoretic viewpoint.  Now, we consider $\Gamma$-equivariant homology theories as functors $\bfE\colon\orbf\rightarrow\spectra$.  Technically, to avoid functorial problems one must take composite functors through the categories $C^\ast$-$\mathbf{Cat}$ and $\mathbf{Groupoids}$.  We do not concern ourselves with this complication and refer the reader to \cite{DL1998} and \cite{DP03}.

Instead we will take for granted that there is a composite functor
\[\KO\colon\orbf\rightarrow\spectra \]
which satisfies $\pi_n\KO(\Gamma/H)=KO_n^{\rm{top}}(C_r^\ast(H))$.  When $\calf=\mathcal{FIN}$ this perspective gives a homotopy theoretic construction of the (real) Baum-Connes assembly map.  Indeed, we have maps
\[B\Gamma_+\wedge\KO\simeq\underset{\mathbf{Or}_{\mathcal{TRV}}(\Gamma)}{\hocolim}\KO\rightarrow\underset{\orbf}{\hocolim}\KO\rightarrow\underset{\mathbf{Or}_{\mathcal{ALL}}(\Gamma)}{\hocolim}\KO\simeq\KO(C^\ast_r(\Gamma;\RR)). \]
The assembly map $\mu_\RR$ is then $\pi_n$ applied to the composite.

\subsection{Group $C^\ast$-algebras and $KK$-theory}
In this section we give a brief outline of Kasparov's $KK$-theory, the material here will not be used elsewhere in the paper.  The theory was introduced by Kasparov in \cite{Kasparov1987,Kasparov1988} in relation to the Novikov Conjecture.  The original formulation of the Baum-Connes Conjecture using $KK$-theory was given in \cite{BaumConnesHigson1994}.

For a $C^\ast$-algebra $A$ define $\mathbf{M}_\infty(A)$ to be the direct limit of sets of $(n\times n)$-matrices over $A$ as $n\rightarrow\infty$.  Similarly, define $\GL_\infty(A)$ to be the direct limit of groups of invertible $(n\times n)$-matrices over $A$.

Topological or operator $K$-theory is a $2$-periodic homology theory of unital $C^\ast$-algebras denoted $K^{\rm{top}}_\ast(-)$.  The zeroth $K$-group of a unital $C^\ast$-algebra $A$ is defined to be the Grothendieck group of the set of projections in $\mathbf{M}_\infty(A)$ up to Murray von Neumann equivalence.  The first $K$-group is defined to be $\GL_\infty(A)/\GL_\infty(A)_0$, where $\GL_\infty(A)_0$ is the path component of the identity.

An alternative formulation is given by Kasparov's bifunctor $KK(-,-)$.  For any two $C^\ast$-algebras $A$ and $B$ there is an abelian group $KK(A,B)$.  An element of $KK(A,B)$ is a homotopy class of $(A,B)$-Fredholm bimodules (see \cite[Section~3]{AparicioJulgValette20} for the precise definition).  The zeroth $K$-group of $A$ from before is recovered as $KK(\CC,A)$ and the first $K$-group is recovered as $KK(C_0(\RR),A)$.

Let $\Gamma$ be a discrete group.  The \emph{reduced $C^\ast$-algebra} of $\Gamma$, denoted $C^\ast_r(\Gamma)$, is the norm closure of the algebra of bounded operators on $\ell^2(\Gamma)$ by the left regular representation of $\Gamma$.  The algebra and its $K$-groups are intimately related with the theory of elliptic operators on manifolds $M$ with fundamental group $\Gamma$.  For more information the reader should consult the survey \cite{AparicioJulgValette20} and the references therein.

\section{Equivariant $KO$-homology of $\SL_3(\ZZ)$}\label{sec.SL3}
\subsection{A classifying space for proper actions}
A model for $X=\underbar{E}\SL_3(\ZZ)$ can be constructed as a $\SL_3(\ZZ)$-equivariant deformation retract of the symmetric space $\SL_3(\RR)/O(3)$.  This construction has been detailed several times in the literature (\cite[Theorem~2]{Sou78}, \cite[Theorem~2.4]{Henn99} or \cite[Theorem~13]{SanchezGarcia:SL}), so rather than detailing it again here, we simply extract the relevant cell complex and cell stabilisers.  Specifically, we follow the notation of S\'anchez-Garc\'ia \cite{SanchezGarcia:SL} and collect the information in Table~\ref{tab.sl3cellcmplx}.

\begin{table}[h]
    \centering
    \[\begin{array}{|c|c|c|c|c|}
    \hline
    \rm Dimension & \rm Cell & \rm Boundary & \rm Stabiliser \\
    \hline
    \hline
    3 & T_1 & -t_1+t_2-t_3+t_4-t_5 & \{1\}\\
    \hline
    \multirow{5}{*}{2} & t_1 & e_1-e_2-e_4 & \ZZ_2\\
    & t_2 & e_4-e_5+e_6 & \{1\}\\
    & t_3 & e_6-e_7+e_8 & \ZZ_2^2\\
    & t_4 & e_1-e_3+e_5+e_8 & \ZZ_2\\
    & t_5 & e_2-e_3+e_6-e_6+e_7 & \ZZ_2\\
    \hline
    \multirow{8}{*}{1} & e_1 & v_1-v_2 & \ZZ_2^2\\
    & e_2 & v_3-v_1 & D_3\\
    & e_3 & v_5-v_1 & D_3\\
    & e_4 & v_3-v_2 & \ZZ_2\\
    & e_5 & v_4-v_2 & \ZZ_2\\
    & e_6 & v_4-v_3 & \ZZ_2^2\\
    & e_7 & v_5-v_3 & D_4\\
    & e_8 & v_5-v_4 & D_4\\
    \hline
    \multirow{5}{*}{0} & v_1 &\multirow{5}{*}{-} & \sym(4)\\
    & v_2 & & D_6\\
    & v_3 & & \sym(4)\\
    & v_4 & & D_4\\
    & v_5 & & \sym(4)\\
    \hline
    \end{array} \]
    \caption{Cell structure and stabilisers of a model for $\underbar{E}\SL_3(\ZZ)$.}
    \label{tab.sl3cellcmplx}
\end{table}

\subsection{Proof of Theorem~\ref{thm.KO.sl3}}
The calculation of the equivariant $KO$-groups will require the following proposition and an analysis of the representation theory of the finite subgroups of $\SL_3(\ZZ)$.  We remark that one could prove a dozen subtle variations on the theme of the following proposition.  However, rather than do this we offer the slogan:  ``Computations with coefficients in $KO_n^\Gamma(-)$ can be greatly simplified by looking for chain maps to the Bredon chain complex with coefficients in $\repc(-)$."

\begin{prop}\label{prop.KtoKO}
Let $\Gamma$ be a discrete group, $\calf=\mathcal{FIN}$ and suppose $X$ is a proper $\Gamma$-CW complex with finitely many $\Gamma$ orbits of cells in each dimension.  Assume that for every cell stabiliser the real, complex and quarternionic character tables are equal, then the Atiyah-Hirzebruch spectral sequence converging to $KO^\Gamma_\ast(X)$ has $E^2$-page isomorphic to
\[E^2_{p,q}=H_p^{\mathcal{FIN}}(X;K^\Gamma_0)\otimes KO_q(\ast)\oplus \Tor^\ZZ_1[H_{p-1}^{\mathcal{FIN}}(X;K^\Gamma_0),KO_q(\ast)] \]
where for $q=0,\dots,7$ we have
\[KO_{q}(\ast)=\ZZ,\quad\ZZ_2,\quad\ZZ_2,\quad 0,\quad\ZZ,\quad0,\quad0,\quad0 \]
and the remaining groups are given by $8$-fold Bott-periodicity.
\end{prop}

Note that the $\Tor$ terms vanish except possibly when $q=1$ or $2$.

\begin{proof}
Since the three character tables are equal, the complexification from $\nu:\repr\to\repc$ and the symplectification from $\sigma:\repc\to\reph$ are isomorphisms.  In the other direction, the complexification from $\eta:\reph\to\repc$ and the realification from $\rho:\repc\to\repr$ correspond to multiplication by $2$.  We will now compute each row of the spectral sequence in turn.

$\mathbf{q=0:}$  We have $E^2_{p,0}=H^\FIN_p(X;KO_0^\Gamma)$ which is exactly equal to $H^\FIN_p(X;\repr)$, the result follows from the isomorphism $H^\FIN_p(X;\repr)\cong H_p^{\mathcal{FIN}}(X;K^\Gamma_0)\otimes \ZZ$ and the vanishing of the $\Tor$ group.

$\mathbf{q=1:}$ The realification $\rho:\repc\to\repr$ is multiplication by $2$, thus the cokernel of the map
\[\rho_\ast:C^\FIN_\ast(X;\repc)\rightarrow C^\FIN_\ast(X;\repr) \]
is the modulo $2$ reduction of $C^\FIN_\ast(X;\repr)$.  Consider $C^\FIN_\ast(X;\repr)$ as a chain complex of abelian groups.  The result follows from the Universal Coefficient Theorem in homology with $\ZZ_2$ coefficients applied to the homology of the chain complex $C^\FIN_\ast(X;\repr)$.

$\mathbf{q=2:}$ The complexification $\eta:\reph\to\repc$ is multiplication by $2$ and $C^\FIN_\ast(X;\repc)$ is isomorphic to $C^\FIN_\ast(X;\repr)$.  The result now follows as in the case $q=1$.

$\mathbf{q=3:}$ Immediate since $KO^\Gamma_3(-)=0$.

$\mathbf{q=4:}$  Since $\nu$ and $\sigma$ are both isomorphisms, their composition gives an isomorphism of Bredon chain complexes $C_\ast^\FIN(X;\repr)\cong C_\ast^\FIN(X;\reph)$.  The result now follows as in the $q=0$ case.

$\mathbf{q=5:}$  Since $\sigma$ is an isomorphism, the cokernel of the map
\[\sigma_\ast: C^\FIN_\ast(X;\repc)\rightarrow C^\FIN_\ast(X;\reph) \]
vanishes.  The result follows.

$\mathbf{q=6:}$  Since $\nu$ is an isomorphism, the cokernel of the map
\[\nu_\ast: C^\FIN_\ast(X;\repr)\rightarrow C^\FIN_\ast(X;\repc) \]
vanishes.  The result follows.

$\mathbf{q=7:}$ Immediate since $KO^\Gamma_7(-)=0$.
\end{proof}

\begin{thm}[Theorem~\ref{thm.KO.sl3}]\label{thm.KO.sl3.body}
Let $\Gamma=\SL_3(\ZZ)$, then for $n=0,\dots,7$ we have
\[KO^\Gamma_{n}(\underbar{\emph{E}}\Gamma)= \ZZ^8,\quad \ZZ_2^8,\quad \ZZ_2^8,\quad 0,\quad \ZZ^8, \quad 0, \quad 0, \quad 0  \]
and the remaining groups are given by $8$-fold Bott-periodicity.
\end{thm}

\begin{proof}
Let $\Gamma=\SL_3(\ZZ)$, $\calf=\mathcal{FIN}$ and $X=\underbar{E}\SL_3(\ZZ)$. We can now complete the calculation for the equivariant $KO$-homology groups.  First, we recap the calculation of the Bredon chain complex with complex representation ring coefficients due to S\'anchez-Garc\'ia.  We have a chain complex
\[\begin{tikzcd}
0 \arrow[r] & \ZZ \arrow[r, "\partial_3"] & \ZZ^{11} \arrow[r, "\partial_2"] & \ZZ^{28} \arrow[r, "\partial_1"] & \ZZ^{26} \arrow[r] & 0
\end{tikzcd} \]
where 
\[\partial_3\sim\begin{bmatrix}1& \mathbf{0}_{1\times10}\end{bmatrix},\quad \partial_2\sim\begin{bmatrix} I_{10} & \mathbf{0}_{10\times18}\\ \mathbf{0}_{10\times1} & \mathbf{0}_{1\times 18}
\end{bmatrix},\quad \text{and}\quad 
\partial_1\sim\begin{bmatrix}I_{18} & \mathbf{0}_{18\times8}\\ \mathbf{0}_{10\times18} & \mathbf{0}_{10\times8} \end{bmatrix}.\]
Therefore, the homology groups of the chain complex are isomorphic to $\ZZ^8$ in dimension $0$ and to $0$ in every other dimension.

Now, the cell stabiliser subgroups of $\SL_3(\ZZ)$ acting on $X$ are isomorphic to $\{1\}$, $\ZZ_2$, $\ZZ_2^2$, $D_3$, $D_4$, $\sym(4)$ and $D_6$.  Each of these satisfies the conditions of the proposition above.  This is easily checked by computing the Schur indicators of each of the irreducible characters of each group. Since the Schur indicator equals $1$ in every case we conclude the three character tables for each group are equal (see for instance \cite[Exercise~3.38]{FultonHarris1991}).  Applying this to the previous calculation we obtain a single non-trivial column when $p=0$ in the Atiyah-Hirzebruch spectral sequence and so it collapses trivially.
\end{proof}

\begin{corollary}[Corollary~\ref{cor.KO.gl3}] \label{cor.KO.gl3.body}
Let $\Gamma=\GL_3(\ZZ)$, then for $n=0,\dots,7$ we have
\[KO^\Gamma_{n}(\underbar{\emph{E}}\Gamma)= \ZZ^{16},\quad \ZZ_2^{16},\quad \ZZ_2^{16},\quad 0,\quad \ZZ^{16}, \quad 0, \quad 0, \quad 0  \]
and the remaining groups are given by $8$-fold Bott-periodicity.
\end{corollary}
\begin{proof}
First, note that the direct product of $\ZZ_2$ with any of the cell stabiliser subgroups of $\SL_3(\ZZ)$ still satisfies the conditions of the Proposition~\ref{prop.KtoKO}.  Now, we may compute the $E^2$-page of the associated Atiyah Hirzebruch spectral sequence by applying the K\"unneth formula \cite[Theorem~3.6]{SanchezGarcia:Coxeter} to the calculation of each row of the $E^2$-page for $\underbar{E}\SL_3(\ZZ)$.  Since the spectral sequence is concentrated in a single column we have isomorphisms
\[ KO^{\GL_3(\ZZ)}_n(\underbar{E}\GL_3(\ZZ))\cong KO^{\SL_3(\ZZ)}_n(\underbar{E}\SL_3(\ZZ))\otimes KO^{\ZZ_2}_n(\ast), \]
from which the result is immediate.
\end{proof}

\section{Equivariant $K$-homology of Fuchsian groups}\label{sec.KthryFuchsian}
In this section we compute the equivariant $K$-homology of every finitely generated Fuchsian group, that is, a finitely generated discrete subgroup of $\PSL_2(\RR)$.  The reason for this apparent detour is that we will later split the groups $\PSLp$ as amalgamated free products of certain Fuchsian subgroups.  Thus, we can use a Mayer-Vietoris type argument to compute their $K$-homology.

Note that Theorem~\ref{thm.EquiKthryFuchsian}\ref{thm.EquiKthryCocoFuchsian} was computed in \cite{luck2000computations} along with a more general result for cocompact NEC groups.  Moreover, their integral cohomology was determined by the author in \cite{Hughes2019cohomology}.  An introduction to Fuchsian groups is provided by \cite{Katok1992}.

 The computation is made easier by the fact that every finitely generated Fuchsian group is described by piece of combinatorial data called a {\it signature} \cite[Chapter 4.3]{Katok1992}. Indeed, a Fuchsian group of {\it signature} $[g,s;m_1,\dots,m_r]$ has presentation
\[\left\langle a_1,\dots,a_{2g},c_1,\dots,c_s,d_1.\dots,d_r\ |\ \prod_{i=1}^g[a_i,a_{g+i}]\prod_{j=1}^rd_j\prod_{k=1}^sc_k=d_1^{m_1}=\dots=d_r^{m_r}=1 \right\rangle \]
and acts on the hyperbolic plane $\RR\textbf{H}^2$ with a $4g+2s+2r$ sided fundamental polygon.  The tessellation of the polygon under the group action has $1+s+r$ orbits of vertices, $s$ of which are on the boundary $\partial\RR\textbf{H}^2$, $2g+s+r$ orbits of edges and $1$ orbit of faces.  All edge and face stabilisers are trivial.  All vertex stabilisers are trivial except for $r$ orbits of vertices, each of which is stabilised by some $\ZZ_{m_j}$.  Note that if $s=0$ we say $\Gamma$ is \emph{cocompact}.  

The signature also describes a quotient $2$-orbifold which is homeomorphic to a genus $g$ surface with $s$ points removed.  The orbifold data is then given by the $r$ marked points, each corresponding to one of the $m_j$, or equivalently a maximal conjugacy of finite subgroups. 

If $r=0$, we do not write any $m_j$ in the signature. In which case $\Gamma$ has signature $[g,s;]$, is torsionfree, and isomorphic to either the fundamental group of a genus $g$ surface, or a free group of rank $2g+s-1$.

\begin{thm}\label{thm.EquiKthryFuchsian}
Let $\Gamma$ be a Fuchsian group of signature $[g,s;m_1,\dots,m_r]$.
\begin{enumerate}[label=(\alph*)]
\item \label{thm.EquiKthryCocoFuchsian} If $s=0$ then, \[K^\Gamma_n(\underbar{\emph{E}}\Gamma)=K_n(C_r^\ast(\Gamma))=
\begin{cases}
\ZZ^{2-r+\sum_{j=1}^rm_j} & n\text{ even,}\\
\ZZ^{2g} & n\text{ odd.}
\end{cases}\]

\item \label{thm.EquiKthryCofinFuchsian} If $s>0$ then, \[K^\Gamma_n(\underbar{\emph{E}}\Gamma)=K_n(C_r^\ast(\Gamma))=
\begin{cases}
\ZZ^{1-r+\sum_{j=1}^rm_j} & n\text{ even,}\\
\ZZ^{2g+s-1} & n\text{ odd.}
\end{cases}\]
\end{enumerate}
\end{thm}
\begin{proof}[Proof of \ref{thm.EquiKthryCocoFuchsian}]
Let $\Gamma$ be a Fuchsian group of signature $[g,s;m_1,\dots,m_r]$ with $s=0$ and $\mathcal{F}=\FIN$. Since $\Gamma$ satisfies the Baum-Connes conjecture \cite{HigsonKasparov1997} it is enough to compute the equivariant $K$-homology.  The hyperbolic plane with the induced cell structure of the $\Gamma$ action is a model for $\underbar{E}\Gamma$ (see for instance \cite{Macbeath1967}).  Recall that the cell structure has $r+1$ orbits of vertices, $2g+r$ orbits of edges and exactly $1$ orbit of $2$-cells.  One vertex $v_0$ is stabilised by the trivial group and for $j=1,\dots,r$ the vertex $v_j$ is stabilised by $\ZZ_{m_j}$.  Thus, we have a Bredon chain complex
\[\begin{tikzcd} 
0 & \arrow[l] \ZZ\oplus \left(\bigoplus_{j=1}^r \repc(\ZZ_{m_j})\right) & \arrow[l,"\partial_1"'] \ZZ^{2g+r} & \arrow[l,"\partial_2"'] \ZZ & \arrow[l] 0,
\end{tikzcd}\]
and substituting in $\repc(\ZZ_{m_j})=\ZZ^{m_j}$ we obtain
\[\begin{tikzcd} 
0 & \arrow[l] \ZZ \oplus \left(\bigoplus_{j=1}^r \ZZ^{m_j}\right) & \arrow[l,"\partial_1"'] \ZZ^{2g+r} & \arrow[l,"\partial_2"'] \ZZ & \arrow[l] 0.
\end{tikzcd}\]

We fix the following basis for each chain group:  In degree $0$ we have generators $x_{j,l}$, for $j=1,\dots,r$ and $l=1,\dots,m_j$, and the generator $z$.  In degree $1$ we have $a_1,\dots,a_{2g}$ and $y_1,\dots,y_r$, and in degree $2$, the generator $w$.  An easy calculation yields that $\partial_2(w)=0$, $\partial_1(a_i)=0$, and $\partial_1(y_j)=\sum_{l=1}^{m_j}x_{j,l} - z$.  Thus,
\[H_n^{\mathcal{F}}(\underbar{E}\Gamma;\repc)= \begin{cases}
\ZZ^{1+\sum_{j=1}^r(m_j-1)} & \text{if }n=0;\\
\ZZ^{2g} & \text{if }n=1;\\
\ZZ & \text{if }n=2;\\
0 & \text{otherwise}.
\end{cases} \]

The result now follows from the collapsed Atiyah-Hirzebruch spectral sequence given in \cite[Theorem~5.27]{MislinNotes} and we obtain $K^\Gamma_0(\underbar{E}\Gamma)=H^{\mathcal{F}}_0(\underbar{E}\Gamma;\repc)\oplus H^{\mathcal{F}}_2(\underbar{E}\Gamma;\repc)$ and $K^\Gamma_1(\underbar{E}\Gamma)=H^{\mathcal{F}}_1(\underbar{E}\Gamma;\repc)$.
\end{proof}

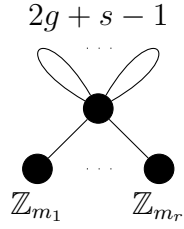
\begin{figure}[ht]
\[
\begin{tikzpicture}[baseline=(current bounding box.center)]
\node   [style=blackNode] (a) at (0,0) {};
\node [style=blackNode] (m1) at (-0.8,-0.8) {};
\node [] (m1label) at (-0.8,-1.3) {$\ZZ_{m_1}$};
\draw (a) to (m1);

\node [style=blackNode] (mr) at (0.8,-0.8) {};
\node [] (mrlabel) at (0.8,-1.3) {$\ZZ_{m_r}$};
\draw (a) to (mr);

\draw[loosely dotted] (-0.15,-0.8) to (0.2,-0.8);

\path[scale=2.5] 
        (a) edge [out= 150, in= 120, distance=5mm] (a)
        (a) edge [out= 60, in= 30, distance=5mm] (a);
\draw[loosely dotted] (-0.15,0.8) to (0.2,0.8);
\node   [] (a) at (0,1.2) {$2g+s-1$};
\end{tikzpicture}
\]
    \caption{A graph of groups for a non-cocompact Fuchsian group.}
    \label{fig.GoG.CofinFuch}
\end{figure}

\begin{proof}[Proof of \ref{thm.EquiKthryCofinFuchsian}]
Let $\Gamma$ be a Fuchsian group of signature $[g,s;m_1,\dots,m_r]$ with $s>0$ and let $\mathcal{F}=\FIN$.  In this case we can rearrange the presentation of $\Gamma$ such that we have a splitting of $\Gamma$ as an amalgamated free product $\Gamma\cong\ZZ^{s-1}\ast\ZZ_{m_1}\ast\dots\ast\ZZ_{m_r}$.  Now, $\Gamma$ splits as a finite graph of finite groups (Figure~\ref{fig.GoG.CofinFuch}) and it is easy to see that the Bass-Serre tree of $\Gamma$ is a model for $\underbar{E}\Gamma$.

We will first compute the Bredon homology $H_\ast^{\mathcal{F}}(\underbar{E}\Gamma;\repc)$ with coefficients in the representation ring, then apply the equivariant Atiyah-Hirzebruch spectral sequence. We have a Bredon chain complex
\[\begin{tikzcd} 
0 & \arrow[l] \ZZ\oplus \left(\bigoplus_{j=1}^r \repc(\ZZ_{m_j}) \right) & \arrow[l,"\partial"'] \ZZ^{2g+s-1} & \arrow[l] 0,
\end{tikzcd}\]
substituting in $\repc(\ZZ_{m_j})=\ZZ^{m_j}$ we obtain
\[\begin{tikzcd} 
0 & \arrow[l] \ZZ \oplus \left(\bigoplus_{j=1}^r \ZZ^{m_j}\right) & \arrow[l,"\partial"'] \ZZ^{2g+s-1} & \arrow[l] 0.
\end{tikzcd}\]

Let the first non-zero term have generating set $\langle x_{j,l},z\ |\ j=1,\dots,r,\ l=1,\dots,m_j\rangle$ and the second term $\langle a_1, \dots, a_{2g}, c_1,\dots,c_{s-1},d_1,\dots,d_r\rangle$.  It is easy to see the differential $\partial$ is given by $\partial(a_i)=\partial(b_i)=\partial(c_k)=0$ and $\partial(d_j)=\sum_{l=1}^{m_j}x_{j,l} - z$.  It follows that $H^{\mathcal{F}}_0(\underbar{E}\Gamma;\repc)=\ZZ^{1+\sum_{j=1}^r(m_j-1)}$, $H^{\mathcal{F}}_1(\underbar{E}\Gamma;\repc)=\ZZ^{2g+s-1}$ and $0$ otherwise.  The result now follows from the collapsed Atiyah-Hirzebruch spectral sequence \cite[Theorem~5.27]{MislinNotes}. In particular, we have $K^\Gamma_n(\underbar{E}\Gamma)=H^{\mathcal{F}}_n(\underbar{E}\Gamma;\repc)$ for $n=0,1$.
\end{proof}

\section{Computations for $\PSLp$ and $\SLp$}\label{sec.SLpKthry}
\subsection{Preliminaries}
In an abuse of notation, throughout this section we will denote the image $\{\pm A\}$ of a matrix $A\in\SL_2(\RR)$ in $\PSL_2(\RR)$ by the matrix $A$.  Recall that for $p$ a prime we have $\PSL_2(\ZZ[\frac{1}{p}])=\PSL_2(\ZZ)\ast_{\Gamma_0(p)}\PSL_2(\ZZ)$, where $\Gamma_0(p)$ is the level $p$ Hecke principle congruence subgroup (see for instance Serre's book ``Trees" \cite{SerreTrees}).  The amalgamation is specified by two embeddings of the congruence subgroup $\Gamma_0(p)$ into $\PSL_2(\ZZ)$.  The first is given by
\[\Gamma_0(p):=\left\{\begin{bmatrix} a&b\\ c&d \end{bmatrix} \in \PSL_2(\ZZ) \colon c\equiv0\pmod{p}\right\} \]
and the second via
\[\begin{bmatrix} a&b\\ c&d \end{bmatrix}\mapsto\begin{bmatrix} a&pb\\ p^{-1}c&d\end{bmatrix}. \]

In light of this we will collect some facts about each of the groups in the amalgamation.  We begin by recording (Table~\ref{tab:BredonPGamma0}) the Fuchsian signatures and the associated Bredon homology for each of the groups $\Gamma_0(p)$ and $\PSL_2(\ZZ)$.  Note that when $p\equiv11\pmod{12}$ the group $\Gamma_0(p)$ is free.

\begin{lemma}
The signatures and Bredon homology groups of $\Gamma_0(p)$ and $\PSL_2(\ZZ)$ are given in Table~\ref{tab:BredonPGamma0}.
\end{lemma}

\begin{table}[h]
    \centering
    \begin{tabular}{|c|c|c|c|}
    \hline
    $p$ & Signature of $\Gamma_0(p)$ & $H^\mathcal{F}_0(\underbar{E}\Gamma_0(p);R_\CC)$ & $H^\mathcal{F}_1(\underbar{E}\Gamma_0(p);R_\CC)$ \rule{0pt}{12pt} \\ \hline \hline
    $2$ & $[0,2;2]$ & $\ZZ^2$ & $\ZZ$ \\
    $3$ & $[0,2;3]$ & $\ZZ^3$ & $\ZZ$ \\
    $p\equiv 1\pmod{12}$ & $[0,\frac{1}{6}(p-7)+1;2,2,3,3]$ & $\ZZ^7$ & $\ZZ^{\frac{1}{6}(p-7)}$ \\
    $p\equiv 5\pmod{12}$ & $[0,\frac{1}{6}(p+1)+1;2,2]$ & $\ZZ^3$ & $\ZZ^{\frac{1}{6}(p+1)}$ \\
    $p\equiv 7\pmod{12}$ & $[0,\frac{1}{6}(p-1)+1;3,3]$ & $\ZZ^5$ & $\ZZ^{\frac{1}{6}(p-1)}$ \\
    $p\equiv 11\pmod{12}$ & $[0,\frac{1}{6}(p+7)+1;]$ & $\ZZ$ & $\ZZ^{\frac{1}{6}(p+7)}$ \\
    \hline
    \hline
    $\PSL_2(\ZZ)$ & $[0,1;2,3]$ & $\ZZ^4$ & $0$\\
    \hline
    \end{tabular}
    \caption{Fuchsian signatures and Bredon homology groups of ${\rm P}\Gamma_0(p)$.}
    \label{tab:BredonPGamma0}
\end{table}

\begin{proof}
Let $[0,s;m_1,\dots,m_r]$ be the signature of $\Gamma_0(p)$.  We will first compute the ordinary cohomology groups of $\Gamma_0(p)$, then using these we will deduce the signatures, finally the Bredon homology may then be read off of Theorem~\ref{thm.EquiKthryFuchsian}\ref{thm.EquiKthryCofinFuchsian}.  Our computation of the cohomology will be near identical to the computation in \cite[Section~2]{AN1998}.  The key difference is that the modules in \cite{AN1998} are for the lifts of $\Gamma_0(p)$ in $\SL_2(\RR)$ whereas we are always working with the projectivised groups (see the discussion after \cite[Proposition~2.2]{AN1998}).  Note that the fact the signature of $\PSL_2(\ZZ)$ is $[0,1;2,3]$ is well known.

Let $G=\PSL_2(p)$ and let $Q$ be the subgroup of equivalence classes of matrices with lower left hand entry equal to zero.  Clearly, $Q\cong\ZZ_p\rtimes\ZZ_{\frac{1}{2}(p-1)}$ (unless $p=2,3$ where $Q\cong\ZZ_p$).  Each $\Gamma_0(p)$ fits into a short exact sequence with normal subgroup a congruence subgroup $\Gamma(p)$ isomorphic to a free group and quotient $Q$.

Now, recall \cite[Example~4.2(c)]{SerreTrees} that $\PSL_2(\ZZ)$ acts on a tree $\calt$ with fundamental domain an edge.  Moreover, $G$ acts on $\calt/\Gamma(p)$.  The stabiliser subgroups for both actions are $\ZZ_2$ and $\ZZ_3$ for the vertices and trivial for the edges.  It follows $\Gamma_0(p)$ acts freely on $EQ\times \calt$ and so $EQ\times_Q\calt/\Gamma_0(p)$ is a model for $B\Gamma_0(p)$.

Let $C^\ast$ denote the cellular cochains on the $B$-CW complex $\calt/\Gamma(p)$, then by \cite[Section~2]{AN1998} we have the following isomorphisms of $B$-modules
\[C^0:=\ZZ[G/\ZZ_2]\ |_B\oplus \ZZ[G/\ZZ_3]\ |_B\quad\text{and}\quad C^1:=\ZZ[G]\ |_B. \]
As in \cite[Section~2]{AN1998} we have a long exact sequence
\[\cdots\rightarrow H^n(\Gamma_0(p);\ZZ)\rightarrow H^n(B;C^0)\rightarrow H^n(B;C^1)\rightarrow H^{n+1}(\Gamma_0(p);\ZZ)\rightarrow\cdots \]
Since $C^0$ is a permutation module, $H^1(B;C^0)=0$.  Calculating ranks yields that $H^1(\Gamma_0(p);\ZZ)=\ZZ^{N(p)}$ where
\[N(p)=1,\quad 1,\quad \frac{1}{6}(p-7),\quad \frac{1}{6}(p+1),\quad \frac{1}{6}(p-1),\quad \frac{1}{6}(p+7), \]
ordered as in Table~\ref{tab:BredonPGamma0}.  In \cite[Theorem~1.4(b)]{Hughes2019cohomology} it is shown that a Fuchsian group $\Gamma$ of signature $[0,s;m_1,\dots,m_r]$ has $H^1(\Gamma;\ZZ)\cong \ZZ^{s-1}$.  In particular, we deduce the signature of $\Gamma_0(p)$ must have the form $[0,N(p)+1;m_1,\dots,m_r]$.

Now, $C^1$ is a free $B$-module and so we have an isomorphism $H^{2n}(\Gamma_0(p);\ZZ)\cong H^2(B;C^0)$ for all $n\geq1$.  As in \cite[Proposition~2.3]{AN1998} we obtain that
\[ H^{2n}(\Gamma_0(p);\ZZ)\cong \ZZ_2,\quad \ZZ_3,\quad \ZZ_6^2,\quad \ZZ_2^2,\quad \ZZ_3^2,\quad 0, \] 
ordered as in Table~\ref{tab:BredonPGamma0}.  The result now follows from the following three facts.  Firstly, for a Fuchsian group $\Gamma$ of signature $[0,s;m_1,\dots,m_r]$, each $m_j$ corresponds to a conjugacy class of maximal finite cyclic subgroups $\ZZ_{m_j}$.  Secondly, by the proof of \cite[Theorem~1.4(b)]{Hughes2019cohomology}, each maximal conjugacy class of finite cyclic subgroups $\ZZ_{m_j}$ contributes a $\widetilde{H}^\ast(\ZZ_{m_j};\ZZ)$ summand to $\widetilde{H}^\ast(\Gamma;\ZZ)$.  Thirdly, $\PSL_2(\ZZ)$ and hence $\Gamma_0(p)$ has no elements of order $6$.
\end{proof}

\begin{remark}
Let $\widetilde{\Gamma_0}(p)$ denote the lift of $\Gamma_0(p)$ in $\SL_2(\RR)$.  An alternative computation of $H^\ast(\Gamma_0(p);\ZZ)$ can be achieved by back solving the Lyndon-Hochschild-Serre spectral sequence (see for instance \cite[Chapter VII.6]{BrownBook}) for the group extension $\ZZ_2\rightarrowtail\widetilde{\Gamma_0}(p)\twoheadrightarrow\Gamma_0(p)$ which takes the form
\[E_2^{\ast,\ast}=H^\ast(\Gamma_0(p);H^\ast(\ZZ_2;\ZZ))\Rightarrow H^{\ast}(\widetilde{\Gamma_0}(p);\ZZ) \]
using the cohomology calculations for $\widetilde{\Gamma_0}(p)$ in \cite{AN1998}.
\end{remark}

We shall also record the conjugacy classes of finite order elements of $\Gamma_0(p)$ and $\PSL_2(\ZZ[\frac{1}{p}])$.  Note that the only conjugacy classes of finite subgroups of $\PSL_2(\ZZ)$ are one class of groups isomorphic to $\ZZ_2$ and one to $\ZZ_3$ since $\PSL_2(\ZZ)\cong\ZZ_2\ast\ZZ_3$.  The conjugacy classes of finite subgroups of $\Gamma_0(p)$ can be read off of the signature, there is exactly one of order $m_j$ for each $j=1,\dots,r$.

\begin{lemma}\label{lem.ccl.psl}
The number of conjugacy classes of finite order elements in $\PSL_2(\ZZ[\frac{1}{p}])$ are those given in Table~\ref{tab:FiniteSubgroupsPSL}.
\end{lemma}

\begin{table}[h]
    \centering
    \begin{tabular}{|c|c|c|c|c|c|c|}
    \hline
    & $p=2$ & $p=3$ & $ p\equiv 1\pmod{12}$& $ p\equiv 5\pmod{12}$& $ p\equiv 7\pmod{12}$ & $ p\equiv 11\pmod{12}$\\
    \hline \hline
    Identity & $1$ & $1$ & $1$ & $1$ & $1$ & $1$\\
    Order $2$ & 1 & 2 & 1 & 1 & 2 & 2 \\
    Order $3$ & 4 & 2 & 2 & 4 & 2 & 4  \\ \hline \hline
    Total & $6$ & 5 & 4 & 6 & 5 & 7 \\
    \hline
    \end{tabular}
    \caption{Number of conjugacy classes of finite order elements of $\PSL_2(\ZZ[\frac{1}{p}])$.}
    \label{tab:FiniteSubgroupsPSL}
\end{table}

\begin{proof}
The result follows from the following observation:  If there is a conjugacy of elements of order $2$ (resp. $3$) in $\Gamma_0(p)$, then each of class of elements of order $2$ (resp. $3$) in $\PSL_2(\ZZ)$ fuses in $\PSL_2(\ZZ[\frac{1}{p}])$.  To see this, consider an element in the first copy of $\PSL_2(\ZZ)$, conjugate it to an element in $
\Gamma_0(p)$, and then conjugate it to an element in the other copy of $\PSL_2(\ZZ)$.
\end{proof}

\begin{lemma}\label{lem.BC}
Both $\SL_2(\ZZ[\frac{1}{p}])$ and $\PSL_2(\ZZ[\frac{1}{p}])$ satisfy the Baum-Connes Conjecture.
\end{lemma}
\begin{proof}
Since $\PSL_2(\ZZ[\frac{1}{p}])=\PSL_2(\ZZ)\ast_{\Gamma_0(p)}\PSL_2(\ZZ)$, the Bass-Serre tree of the amalgamation is a locally-finite $1$-dimensional contractible $\PSL_2(\ZZ[\frac{1}{p}])$-CW complex.  Moreover, each of the stabilisers $\Gamma_c$ have ${\rm cd}_\QQ(\Gamma_c)=1$, being a graph of finite groups.  Now, we apply \cite[Corollary~5.14]{MislinNotes} to see that the stabilisers satisfy Baum-Connes and \cite[Theorem~5.13]{MislinNotes} to see that $\PSL_2(\ZZ[\frac{1}{p}])$ does.  The proof is identical for $\SL_2(\ZZ[\frac{1}{p}])$.
\end{proof}

\subsection{Computations}
There is a long exact Mayer-Vietoris sequence for computing the Bredon homology of an amalgamated free product.

\begin{thm}\emph{\cite[Corollary~3.32]{MislinNotes}}
Let $\Gamma=H\ast_LK$ and let $M$ be a Bredon module.  There is a long exact Mayer-Vietoris sequence:
\[
\begin{tikzcd}
\cdots \arrow[r] & H_n^{\mathcal{FIN}}(L;M) \arrow[r]     & H_n^{\mathcal{FIN}}(H;M)\oplus H_n^{\mathcal{FIN}}(K;M) \arrow[d] \\
\cdots           & H_{n-1}^{\mathcal{FIN}}(L;M) \arrow[l] & H_n^{\mathcal{FIN}}(\underbar{\emph{E}}\Gamma;M) \arrow[l]                           
\end{tikzcd}
\]
\end{thm}

We are now ready to compute the $K$-theory of $\PSLp$.

\begin{thm}[Theorem~\ref{thm.equiK.PSL1overp}]\label{thm.equiK.PSL1overp.body}
Let $p$ be a prime and $\Gamma=\PSL_2(\ZZ[\frac{1}{p}])$, then $K_n^{\Gamma}(\underbar{{\emph E}}\Gamma)$ is a free abelian group with rank as given in Table~\ref{tab:PSL1pEquiK}.  Moreover, since the Baum-Connes Conjecture holds for $\Gamma$ we have $K^\Gamma_\ast( \underbar{{\emph E}}\Gamma)\cong K^\textrm{top}_\ast(C^\ast_r(\Gamma))$.
\end{thm}
\begin{proof}
There are $6$ cases to consider, the two cases when $p=2,3$ and the four cases given by $p\equiv1,5,7,11\pmod{12}$.  Let $\Gamma=\PSL_2(\ZZ[\frac{1}{p}])$ and $\mathcal{F}=\FIN$. In each case we have the following long exact Mayer-Vietoris sequence
\[\begin{tikzcd} 
0 \arrow[r] & H_2^{\mathcal{F}}(\Gamma;\repc) \arrow[r] &  H_1^{\mathcal{F}}(\Gamma_0(p);\repc) \arrow[r] & \left(H_1^{\mathcal{F}}(\PSL_2(\ZZ);\repc)\right)^2 \arrow[d] \\
   & \left(H_0^{\mathcal{F}}(\PSL_2(\ZZ);\repc)\right)^2 \arrow[d] & \arrow[l] H_0^{\mathcal{F}}(\Gamma_0(p);\repc)  &  \arrow[l]  H_1^{\mathcal{F}}(\Gamma;\repc)  \\
 & H_0^{\mathcal{F}}(\Gamma;\repc) \arrow[r] & 0.                                      &                       \end{tikzcd}\]

We have computed the Bredon homology groups of $\PSL_2(\ZZ)$ and $\Gamma_0(p)$ in Table~\ref{tab:BredonPGamma0}.  Thus, we can separate the above sequence into two sequences.  Indeed, $H_1^{\mathcal{F}}(\PSL_2(\ZZ);\repc))=0$, so it follows that $H_2^{\mathcal{F}}(\Gamma;\repc)\cong H_1^{\mathcal{F}}(\Gamma_0(p);\repc)$.  The other sequence is then given by the remaining terms.

We will treat the case $p=2$, the other cases proceed identically.  We have $H_2^{\mathcal{F}}(\Gamma;\repc)=\ZZ$ and an exact sequence
\[\begin{tikzcd}  
0 \arrow[r] & H_1^{\mathcal{F}}(\Gamma;\repc) \arrow[r] & \ZZ^2 \arrow[r] & \ZZ^8 \arrow[r] & H_0^{\mathcal{F}}(\Gamma;\repc) \arrow[r] & 0.
\end{tikzcd}\]

We now compute the colimit $H_0^{\mathcal{F}}(\Gamma;\repc)=\textrm{colim}_{\Gamma/H\in\orbf}\repc(H)$.  Since we have a complete description of the conjugacy classes of finite subgroups of $\Gamma$ and the only inclusions are given by $\{1\}\hookrightarrow\ZZ_2$ and $\{1\}\hookrightarrow\ZZ_3$, it follows that $H_0^{\mathcal{F}}(\Gamma;\repc)=\ZZ^6$.  Moreover, for the sequence to be exact, it follows the map $\ZZ^2\rightarrow\ZZ^8$ must be an isomorphism onto the kernel of the first map.  In particular $H_1^{\mathcal{F}}(\Gamma;\repc)=0$.  

The result now follows from the collapsed Atiyah-Hirzebruch spectral sequence given in \cite[Theorem~5.27]{MislinNotes} and we obtain $K^\Gamma_0(\underbar{E}\Gamma)=H_0^{\mathcal{F}}(\Gamma;\repc)\oplus H_2^{\mathcal{F}}(\Gamma;\repc)$ and $K^\Gamma_1(\underbar{E}\Gamma)=H_1^{\mathcal{F}}(\Gamma;\repc)$.  We record the Bredon homology groups for the remaining cases in Table~\ref{tab:PSL1pBredon}, the reader can easily verify these.  Note that they are always torsion-free and so are completely determined by their $\ZZ$-rank.
\end{proof}

\begin{table}[h]
    \centering
\begin{tabular}{|c|c|c|c|c|c|c|}
    \hline
    & $p=2$ & $p=3$ & $ p\equiv 1\pmod{12}$& $ p\equiv 5\pmod{12}$& $ p\equiv 7\pmod{12}$ & $ p\equiv 11\pmod{12}$\\
    \hline \hline
    $n=0$ & $6$ & $5$ & $4$ & $6$ & $5$ & $7$\\
    \hline
    $n=1$ & $0$ & $0$ & $3$ & $1$ & $2$ & $0$\\
    \hline
    $n=2$ & $1$ & $1$ & $\frac{1}{6}(p-7)$ & $\frac{1}{6}(p+1)$ & $\frac{1}{6}(p-1)$ & $\frac{1}{6}(p+7)$ \rule{0pt}{11pt} \\
    \hline
\end{tabular}
    \caption{$\ZZ$-rank of the Bredon homology of $\PSL_2(\ZZ[\frac{1}{p}])$.}
    \label{tab:PSL1pBredon}
\end{table}

The computation for $\SL_2(\ZZ[1/p])$ is almost entirely analogous.  We highlight the differences below.

\begin{thm}[Theorem~\ref{thm.equiK.SL1overp}]\label{thm.equiK.Sl1overp.body}
Let $p$ be a prime and $\Gamma=\SLp$.  Then $K_n^\Gamma(\underbar{\emph{E}}\Gamma)$ is additively isomorphic to the direct sum of two copies of the corresponding equivariant $K$-homology group of $\PSLp$.
\end{thm}
\begin{proof}[Proof (Sketch).]
Let $\calf=\FIN$.  First, we must compute the Bredon homology of the lifts of $\PSL_2(\ZZ)$ and $\Gamma_0(p)$ to $\SL_2(\RR)$.  For this we use the graph of groups in Figure~\ref{fig.GoG.CofinFuch} and note that now every edge group is the same central copy of $\ZZ_2$ and the vertex groups change as follows: The vertices with trivial vertex group now have vertex group the same central copy of $\ZZ_2$.  The vertex groups isomorphic to $\ZZ_2$ are now $\ZZ_4$ and the vertex groups isomorphic to $\ZZ_3$ are now $\ZZ_6$ (each extended by the central $\ZZ_2$).  Computing the Bredon homology we find that in each case it is isomorphic to the direct sum of two copies of the corresponding Bredon homology group in the projective case.

Now, we apply the long exact Mayer-Vietoris sequence to the amalgamated free product decomposition of $\SLp$.  Since $H^\calf_1(\SL_2(\ZZ);\repc)=0$, like in the projective case, the sequence splits into two exact sequences.  The computation of $H^\mathcal{F}_2(\SLp))$ is immediate as before and is additively isomorphic to the direct sum of two copies of $H_2^\mathcal{F}(\PSLp;\repc)$.

To compute the zeroth and first homology groups we will again use the colimit isomorphism $H_0^\mathcal{F}(\Gamma;\repc) = \textrm{colim}_{\Gamma/H\in\orbf}\repc(H)$.  To use this we obtain a count of the total number of conjugacy classes of finite order elements in $\SLp$.  To do this use a near identical argument to Lemma~\ref{lem.ccl.psl} that takes into account the central $\ZZ_2$ subgroup.  It follows that the number of conjugacy classes of elements of finite order in $\SLp$ is equal to twice the number for the corresponding projective group.  It follows that the colimit computation for $H_0^\mathcal{F}(\SLp;\repc)$ is additively isomorphic to the direct sum of two copies of $H_0^\mathcal{F}(\PSLp;\repc)$, where $\Gamma=\SLp$.

From here one computes $H_1^\mathcal{F}(\SLp;\repc)$ in an identical manner to the projective case.  The resulting groups are isomorphic to the direct sum of two copies of the corresponding Bredon homology groups in the projective case.  The result now follows from the collapsed Atiyah-Hirzebruch spectral sequence given in \cite[Theorem~5.27]{MislinNotes} and we obtain $K^\Gamma_0(\underbar{E}\Gamma)=H_0^{\mathcal{F}}(\Gamma;\repc)\oplus H_2^{\mathcal{F}}(\Gamma;\repc)$ and $K^\Gamma_1(\underbar{E}\Gamma)=H_1^{\mathcal{F}}(\Gamma;\repc)$.
\end{proof}

\section{The Unstable Gromov-Lawson-Rosenberg Conjecture}\label{sec.GLR}
Given a smooth closed $n$-manifold $M$ a classical question is to ask whether $M$ admits a Riemannian metric of positive scalar curvature.  In a vast generalisation of the Atiyah-Singer index theorem, Rosenberg \cite{R1986b} exhibits a class in $KO_n^{\rm{top}}(C^\ast_r(\pi_1(M)))$ which is an obstruction to $M$ admitting a metric of positive scalar curvature.

More precisely, let $M$ be a closed spin $n$-manifold and $f\colon M\rightarrow B\Gamma$ be a continuous map for some discrete group $\Gamma$.  Let $\alpha\colon\Omega^{\textrm{Spin}}_n(B\Gamma)\rightarrow KO_n^{\rm{top}}(C_r^\ast(\Gamma))$ be the index of the Dirac operator.  If $M$ admits a metric of positive scalar curvature, then $\alpha[M,f]=0\in KO_n^{\rm{top}}(C_r^\ast(\Gamma))$

\begin{conjecture}[The Unstable GLR Conjecture]
Let $M$ be a closed spin $n$-manifold and $\Gamma=\pi_1(M)$.  If $f\colon M\rightarrow B\Gamma$ is a continuous map which induces the identity on the fundamental groups, then $M$ admits a metric positive scalar curvature if and only if $\alpha[M,f]=0\in KO_n^{\rm{top}}(C_r^\ast(\Gamma))$.
\end{conjecture}

The conjecture has been verified in the case of some finite groups \cite{Stolz1992, RS1994, KS1990, Rosenberg1986}, when the group has periodic cohomology, torsion-free groups for which the dimension of $B\Gamma$ is less than $9$ \cite{JS1998}, and cocompact Fuchsian groups \cite{DP03}.  However, there are counterexamples:  The first is due to Schick \cite{Schick1998} who disproves the conjecture for the direct product $\ZZ^4\times\ZZ_n$ when $n$ is odd; while other instances have been constructed in \cite{JS1998}.  For more information on the Unstable GLR Conjecture the reader should consult \cite{JS1998} and the references therein.

\subsection{Proof of Theorem~\ref{thm.GLR}}
We will now prove the conjecture for a large class of groups.  Our proof is structurally similar to the proof by Davis-Pearson \cite{DP03} so we will summarise their method and highlight any differences. 

Let $ko$ be the connective cover of $KO$ with covering map $p$ and let $D$ be the $ko$-orientation of spin bordism.  The map $\alpha$ (from above) is obtained by the following composition
\[\begin{tikzcd}
\Omega_n^{\textrm{Spin}}(B\Gamma) \arrow[r, "D"] & ko_n(B\Gamma) \arrow[r, "p"] & KO_n(B\Gamma) \arrow[r, "\mu_\RR"] & KO_n^{\rm{top}}(C_r^\ast(\Gamma))
\end{tikzcd} \]
We note that $ko_n(\ast)=0$ for $n<0$ and that $p$ is an isomorphism for $n\geq0$ on the one point space.

Recall from the introduction that a group $\Gamma$ satisfies:
\begin{enumerate}
    \item[(M)] If every finite subgroup is contained in a unique maximal finite subgroup.
    \item[(NM)] If $M$ is a maximal finite subgroup of $\Gamma$, then the normaliser $N_\Gamma(M)$ of $M$ is equal to $M$.
    \item[(BC)] If $\Gamma$ satisfies the Baum-Connes conjecture.
    \item[(PFS)] If all maximal finite subgroups of $\Gamma$ are odd order and have periodic cohomology.
\end{enumerate}

\begin{prop}
Let $\Gamma$ be a group satisfying (BC), (M), and (NM).  Let $\Lambda$ be a set of conjugacy classes of maximal finite subgroups of $\Gamma$.  There is a commutative diagram with exact rows
\[\begin{tikzcd}
\widetilde{KO}_{n+1}(\underbar{\emph{B}}\Gamma) \arrow[r] \arrow[d, "id"] &
\underset{(H)\in\Lambda}{\bigoplus}\widetilde{KO}_n(BH) \arrow[r] \arrow[d, "\mu_\RR"] &
\widetilde{KO}_{n}(B\Gamma) \arrow[d, "\mu_\RR"] \arrow[r] &
\widetilde{KO}_{n}(\underbar{\emph{B}}\Gamma) \arrow[d, "id"] \\
\widetilde{KO}_{n+1}(\underbar{\emph{B}}\Gamma) \arrow[r]                  & \underset{(H)\in\Lambda}{\bigoplus}\widetilde{KO}_n^{\rm{top}}(C_r^\ast(H)) \arrow[r]         & \widetilde{KO}_{n}^{\rm{top}}(C_r^\ast(\Gamma)) \arrow[r]        & \widetilde{KO}_{n}(\underbar{\emph{B}}\Gamma).    
\end{tikzcd}\]
\end{prop}
\begin{proof}
First, since $\Gamma$ satisfies (BC), (M) and (NM) by either \cite[Corollary~3.13]{pchain} or the proof of \cite[Theorem~4.1]{pchain} for any constant functor $\bfE_c\colon\orbf\rightarrow\spectra$ by $\Gamma/H\mapsto \bfE$ there are long exact sequences
\[\begin{tikzcd}
\cdots \arrow[r] & \underset{(H)\in\Lambda}{\bigoplus}H_n(BH;\bfE) \arrow[r] & \left(\underset{(H)\in\Lambda}{\bigoplus}\pi_n(\bfE)\right) \oplus H_n(B\Gamma;\bfE) \arrow[r] & H_n(\underbar{B}\Gamma;\bfE) \arrow[r] & \cdots
\end{tikzcd} \]
and
\[\begin{tikzcd}
\cdots \arrow[r] &  \underset{(H)\in\Lambda}{\bigoplus}\widetilde{H}_n(BH;\bfE) \arrow[r] & \widetilde{H}_n(B\Gamma;\bfE) \arrow[r] & \widetilde{H}_n(\underbar{B}\Gamma;\bfE) \arrow[r] & \cdots
\end{tikzcd} \]
The result then follows a diagram chase exactly as in \cite[Proposition~4]{DP03}, taking $\bfE=\KO$ and the isomorphism
\[\pi_n\left(\underset{\orbf}{\hocolim}(\bfE_c)\right)\cong H_n(\underbar{B}\Gamma;\bfE). \]
\end{proof}

\begin{thm}[Theorem~\ref{thm.GLR}]\label{thm.GLR.body}
Let $\Gamma$ be a group satisfying (BC), (M), (NM) and (PFS).  If $\underbar{\emph{B}}\Gamma$ is finite and has dimension at most $9$, then the Unstable Gromov-Lawson-Rosenberg Conjecture holds for $\Gamma$.
\end{thm}
\begin{proof}
Let $\ko$ be the spectrum of the connective cover of $\KO$.  Via the cover we obtain a natural transformation $p:\ko_c\rightarrow\KO_c$ of constant $\orbf$-$\spectra$.  From the previous proposition we obtain a commutative diagram
\begin{equation}\label{eqn.comdia}\tag{$\dagger$}\begin{tikzcd}
\widetilde{ko}_{n+1}(\underbar{B}\Gamma) \arrow[r] \arrow[d, "p"] & \underset{(H)\in\Lambda}{\bigoplus}\widetilde{ko}_n(BH) \arrow[r] \arrow[d, "\mu_\RR\circ p"] & \widetilde{ko}_{n}(B\Gamma) \arrow[d, "\mu_\RR\circ p"] \arrow[r] & \widetilde{ko}_{n}(\underbar{B}\Gamma) \arrow[d, "p"] \\
\widetilde{KO}_{n+1}(\underbar{B}\Gamma) \arrow[r]                & \underset{(H)\in\Lambda}{\bigoplus}\widetilde{KO}_n^{\rm{top}}(C_r^\ast(H)) \arrow[r]          & \widetilde{KO}_{n}^{\rm{top}}(C_r^\ast(\Gamma)) \arrow[r]          & \widetilde{KO}_{n}(\underbar{B}\Gamma).              
\end{tikzcd} \end{equation}

By Joachim-Schick \cite[Lemma~2.6]{JS1998} $p$ is an isomorphism for $n\geq 6$ and an injection for $n=5$.  Now, suppose that $n\geq 5$ so we are in the setting of the GLR conjecture.  Consider an element \[\beta\in K:=\Ker\left(\mu_\RR\circ p:ko_n(B\Gamma)\rightarrow KO_n(C_r^\ast(\Gamma;\RR)\right)\]
and note that $K\cong \Ker(\mu_\RR\circ p:\widetilde{ko}_n(B\Gamma)\rightarrow \widetilde{KO}_n^{\rm{top}}(C_r^\ast(\Gamma))$.  Combining the diagram \eqref{eqn.comdia} with the isomorphism $p\colon \widetilde{ko}_n(\underbar{B}\Gamma)\rightarrow\widetilde{KO}_n^{\rm{top}}(\underbar{B}\Gamma)$ for $n\geq 6$ (injection for $n=5$), we can deduce that there exists
\[\gamma\in\Ker\left(\underset{(H)\in\Lambda}{\bigoplus}ko_n(BH)\rightarrow \underset{(H)\in\Lambda}{\bigoplus}KO_n(BH) \right) \]
which maps to $\beta$. 

For a group $L$ let $ko^+_n(BL)$ be the subgroup of $ko_n(BL)$ given by $D[M,f]$ where $M$ is a positively curved spin manifold and $f$ is a continuous map.  In \cite{BGS1997} the authors prove for any finite group of odd order with periodic cohomology $H$, that $ko^+_n(BH)=\Ker(\mu_\RR\circ p: ko_n(BH)\rightarrow KO_n^{\rm{top}}(C_r^\ast(H))$.  Thus, we have $\gamma\in ko^+_n(BH)$ and $\beta\in ko^+_n(B\Gamma)$.  Now, in \cite{Stolz1995} it is proven that if $D[M,f]\in ko^+_n(BG)$, then $M$ admits a metric of positive scalar curvature.  In particular, we are done.
\end{proof}

\begin{remark}\label{remark.GLR}
It is unclear whether the assumption that the finite subgroups having odd order can be dropped.  Indeed, it was pointed out to the author by J.F. Davis that the statement of \cite[Corollary~2.2]{BGS1997} contains a misprint.  One should instead (in the notation of \cite{BGS1997}) define $\mathcal{Y}_n(B\pi)$ to be the kernel of $A\circ p$ restricted to the subgroup $D(\Omega^n(B\pi))\subseteq ko_n(B\pi)$. After making this correction, the statements and proofs in the paper are correct.  This error caused a mistake in the main theorem of \cite{DP03} which is only correct if one restricts to Fuchsian groups whose torsion only has odd order.
\end{remark}

\begin{corollary}
Let $M$ be a connected closed spin $n$-manifold and let $\Gamma=\pi_1(M)$ be a group satisfying (M), (NM) and (PFS).  Suppose the assembly map $\mu_\RR$ is injective and $\underbar{\emph{B}}\Gamma$ is finite of dimension $N$.  If $n\geq\max\{5,N-4\}$, then $M$ admits a metric with positive scalar curvature if and only if $\alpha[M,f]=0$.
\end{corollary}

\subsection{Some examples}
In this section we will detail some applications of Theorem~\ref{thm.GLR} to various families of groups.  These results are new whenever the groups involved are infinite and have torsion.

\subsubsection{Graphs of groups}
In \cite[Theorem~3.1]{Saldana2020} it is shown that the fundamental groups of graphs of groups with vertex groups satisfying (M) and (NM) and with torsion-free edge groups, satisfy (M) and (NM).  It follows that we have the following combination theorem:

\begin{corollary}
The $\Gamma$ be a finitely presented fundamental group of a graph of groups such that the vertex groups satisfy (BC), (M) and (NM) and the edge groups are torsion-free and satisfy (BC).  If the vertex groups satisfy (PFS) and $\underbar{\emph{B}}\Gamma$ has dimension at most $9$, then $\Gamma$ satisfies the Unstable GLR Conjecture.
\end{corollary}

\subsubsection{$3$-manifold groups}  
In \cite[Section~3.3]{Saldana2020} it is shown that $3$-manifold groups satisfy (M) and (NM) and there is a well known classification of finite subgroups of orientable connected $3$-manifold groups.  These are exactly the groups which act freely on the $3$-sphere and so have periodic cohomology by \cite[Chapter VI.9]{BrownBook}. In \cite{MattheyOyonoPitsch08} it is shown $3$-manifold groups satisfy (BC).  Applying Theorem~\ref{thm.GLR} we obtain the following result (which is new whenever $\Gamma$ is infinite and contains torsion):

\begin{corollary}
Let $M$ be a closed orientable connected $3$-manifold
with fundamental group $\Gamma$.  If $\Gamma$ has no elements of order $2$ then $\Gamma$ satisfies the Unstable GLR Conjecture.
\end{corollary}

\subsubsection{One-relator groups}
In \cite[Page~32]{pchain} it is shown that one-relator groups satisfy (BC), (M) and (NM) and admit a two dimensional model for $\underbar{E}\Gamma$.  Applying Theorem~\ref{thm.GLR} we obtain the following result (which to the authors knowledge is new whenever $\Gamma$ is infinite and contains torsion):

\begin{corollary}
Let $\Gamma=\langle X\ |\ w\rangle$ be a finitely generated one-relator group and suppose $w$ has odd order when interpreted in $\Gamma$, then $\Gamma$ satisfies the Unstable GLR Conjecture.
\end{corollary}

\subsubsection{Hilbert modular groups}
Let $k$ be a totally real number field of degree $n$ and $\calo_k$ be its ring of integers. The Hilbert modular group of $k$ is defined to be $\PSL_2(\calo_k)$ and is a lattice in $\PSL_2(\RR)^n$. Note that if $k=\QQ$ then we recover the classical modular group $\PSL_2(\ZZ)$. Properties (BC), (M) and (NM) are given in \cite[Lemma 4.3]{BS2016}.  Applying Theorem~\ref{thm.GLR} we obtain the following new result:

\begin{corollary}
Let $k$ be a totally real number field with degree less than or equal to $4$.  Let $\Gamma\leq\PSL_2(\calo_k)$ be finitely presented.  If all finite subgroups of $\Gamma$ are cyclic of odd order then $\Gamma$ satisfies the Unstable GLR Conjecture.
\end{corollary}
\begin{proof}
This follows from the fact every finite subgroup of $\PSL_2(\RR)^n$ is a product of finite cyclic groups. 
\end{proof}

\subsubsection{Subgroups of $\PSLp$ for $p\equiv11\pmod{12}$}
In this section we will prove the result that many subgroups of $\PSLp$ for $p\equiv11\pmod{12}$ satisfy the conditions of Theorem~\ref{thm.GLR}.  The result is new whenever the subgroup is infinite, has torsion, and is not isomorphic to a Fuchsian group.  We will also compute the $KO$-theory of $C_r^\ast(\PSLp;\RR)$.

\begin{corollary}
Let $p\equiv11\pmod{12}$ and let $\Gamma<\PSLp$ be finitely presented.  If $\Gamma$ has no elements of order $2$, then $\Gamma$ satisfies the Unstable GLR Conjecture.
\end{corollary}

There are many finite index subgroups of $\PSLp$ for $p\equiv11\pmod{12}$ satisfying the hypothesis of the corollary.  Indeed, by the amalgamated free product decomposition, $\PSLp$ is generated by the two subgroups isomorphic to $\PSL_2(\ZZ)\cong\ZZ_2\ast\ZZ_3$.  Thus, $\PSLp$ is a four generated group $\langle a,b,c,d\rangle$ where $a$ and $c$ have order $2$, and $b$ and $d$ have order $3$.  The kernel of the homomorphism $\phi:\PSLp\rightarrow\ZZ_2$ by $a,c\mapsto1$ and $b,d\mapsto 0$ is a finite index subgroup of $\PSLp$ with no $2$-torsion.

\begin{proof}
The proof follows from applying Theorem~\ref{thm.GLR} to the observation that every finite subgroup of $\PSLp$ is cyclic and hence has periodic cohomology and the following lemma.  Note that $\Gamma$ is necessarily a proper subgroup since $\PSLp$ always contains elements of order $2$.
\end{proof}

\begin{lemma}
Let $p$ be a prime, then $\PSLp$ satisfies (M).  Moreover, if $p\equiv11\pmod{12}$ then $\PSLp$ satisfies (NM).
\end{lemma}
\begin{proof}
Since each non-trivial finite subgroup of $\Gamma$ is of order $2$ or $3$ it is obvious that $\Gamma$ satisfies (M).  Now, assume $p\equiv11\pmod{12}$ and note that $\PSL_2(\ZZ)\cong\ZZ_2\ast\ZZ_3$ satisfies (NM).  Recall the amalgamated free product decomposition, $\PSLp=\PSL_2(\ZZ)\ast_{\Gamma_0(p)}\PSL_2(\ZZ)$.  The amalgamated subgroup $\Gamma_0(p)$ is torsion-free so we may apply \cite[Theorem~3.1]{Saldana2020}.
\end{proof}

An alternative direct proof of the calculations of the $K$-groups of $C_r^\ast(\PSLp)$ when $p\equiv11\pmod{12}$ is as follows.  Note that this bypasses the computation of the Bredon homology but does not give us a way to compute either invariant for $\SLp$.

\begin{lemma}
Let $\Gamma=\PSLp$, then $\underbar{\emph{B}}\Gamma\simeq \bigvee_{b_1(\Gamma_0(p))}S^2$ 
\end{lemma}
\begin{proof}
We have 
\begin{align*}
    \underbar{B}\Gamma &\simeq \underset{\Top}{\hocolim}\left((\underbar{E}\PSL_2(\ZZ)\times_{\PSL_2(\ZZ)}\Gamma) \leftarrow (\underbar{E}\Gamma_0(p)\times_{\Gamma_0(p)}\Gamma) \rightarrow (\underbar{E}\PSL_2(\ZZ)\times_{\PSL_2(\ZZ)}\Gamma) \right)/\Gamma,\\
    &\simeq \underset{\Top}{\hocolim}\left((\underbar{E}\PSL_2(\ZZ)\times_{\PSL_2(\ZZ)}\Gamma)/\Gamma \leftarrow (\underbar{E}\Gamma_0(p)\times_{\Gamma_0(p)}\Gamma)/\Gamma \rightarrow (\underbar{E}\PSL_2(\ZZ)\times_{\PSL_2(\ZZ)}\Gamma)/\Gamma \right).\\
    \intertext{Since $\underbar{E}\PSL_2(\ZZ)/\PSL_2(\ZZ)$ is an interval and $\underbar{E}\Gamma_0(p)/\Gamma_0(p)$ is a finite graph, we have}
    \underbar{B}\Gamma&\simeq \underset{\Top}{\hocolim}\left(I\leftarrow \bigvee_{b_1(\Gamma_0(p))}S^1\rightarrow I\right),
\end{align*}
but $I$ is contractible, so up to homotopy this becomes a suspension of a wedge of circles.  In particular, $\underbar{B}\Gamma\simeq \bigvee_{b_1(\Gamma_0(p))}S^2$.
\end{proof}

\begin{thm}
Let $p\equiv11\pmod{12}$ and let $\Gamma=\PSLp$, then \[K_0^{\rm{top}}(C_r^\ast(\Gamma))=\ZZ^{7+\frac{1}{6}(p+7)}\quad \text{and}\quad K_1^{\rm{top}}(C_r^\ast(\Gamma))=0.\]
\end{thm}
\begin{proof}
Let $\Lambda$ be a set of representatives of finite subgroups of $\Gamma$.  By \cite[Theorem~4.1(a)]{pchain} we have a short exact sequence
\[0\rightarrow \underset{(H)\in\Lambda}{\bigoplus}\widetilde{K}_n^{\rm{top}}(C_r^\ast(H))\rightarrow K_n^{\rm{top}}(C_r^\ast(\Gamma))\rightarrow K_n(\underbar{B}\Gamma)\rightarrow 0. \]
The only nontrivial part now is computing $K_n(\underbar{B}\Gamma)$, but we have already shown that $\underbar{B}\Gamma$ is homotopy equivalent to a wedge of $\bigvee_{b_1(\Gamma_0(p))}S^2$, i.e. a wedge of $2$-spheres.  Thus, we can simply apply the homological Atiyah-Hirzebruch spectral sequence (which collapses trivially) to obtain that $K_0(\underbar{B}\Gamma)=\ZZ^{\frac{1}{6}(p+7)+1}$ and $K_1(\underbar{B}\Gamma)=0$.
\end{proof}

A near identical argument can be used to compute the $KO$-groups of $C_r^\ast(\PSLp)$ when $p\equiv11\pmod{12}$.

\begin{thm}
Let $p\equiv11\pmod{12}$ be a prime and $\Gamma=\PSLp$.  Except for an extension problem in dimensions congruent to $1,3$ and $4$ modulo $8$, we have for $n=0,\dots,7$ that
\[KO_n^{\rm{top}}(C_r^\ast)=\ZZ^5,\quad \ZZ_2^3,\quad \ZZ^{2+\frac{1}{6}(p+7)}\oplus\ZZ_2^3,\quad \ZZ_2^{\frac{1}{6}(p+7)},\quad \ZZ^5\oplus\ZZ_2^{\frac{1}{6}(p+7)},\quad 0,\quad \ZZ^{2+\frac{1}{6}(p+7)},\quad 0 \]
and the remaining groups are given by $8$-fold Bott-periodicity.
\end{thm}
\begin{proof}
Let $\Lambda$ be a set of representatives of finite subgroups of $\Gamma$.  As before, by \cite[Theorem~4.1(a)]{pchain} we have a short exact sequence
\[0\rightarrow \underset{(H)\in\Lambda}{\bigoplus}\widetilde{KO}_n^{\rm{top}}(C_r^\ast(H))\rightarrow KO_n^{\rm{top}}(C_r^\ast(\Gamma))\rightarrow KO_n(\underbar{B}\Gamma)\rightarrow 0. \]
Now, $KO_n(\underbar{B}\Gamma)\cong KO_n(\ast)\oplus KO_{n-2}(\ast)^{\frac{1}{6}(p+7)}$ and the groups $KO_n^{\ZZ_m}(\ast)\cong KO^{\rm{top}}_n(C_r^\ast(\ZZ_m))$ are given in \cite[Section~2.1]{fuentesrumi2019equivariant}.
\end{proof}

Our methods leave open the following.

\begin{question}
Let $p$ be a prime, then does $\PSL_2(\ZZ[\frac{1}{p}])$ satisfy the Unstable GLR Conjecture?  What about a lattice in $\PSL_2(\RR)\times\PSL_2(\QQ_p)$?
\end{question}

\end{document}